\documentclass{amsart}
\usepackage{graphicx, color}
\usepackage{amscd}
\usepackage{amsmath,empheq}
\usepackage{amsfonts}
\usepackage{amssymb}
\usepackage{mathrsfs}
\usepackage[all]{xy}
\newtheorem{theorem}{Theorem}

\newtheorem{prop}[theorem]{Proposition}

\newenvironment{proof-sketch}{\noindent{\bf Sketch of Proof}\hspace*{1em}}{\qed\bigskip}
\newcommand{\RR}{\mathbb R}
\newcommand{\NN}{\mathbb N}

\newcommand{\di}{\displaystyle}

\renewcommand{\leq}{\leqslant}

\renewcommand{\geq}{\geqslant}
\begin{document}
\title[Positive solutions for superdiffusive mixed problems]{Positive solutions for superdiffusive mixed problems}
%%%%%%%%%%%%%%%%%%%%%%%%%%%%%%%%%%%%%%%%%%%%%%%%%%%%%%%%%%%%%%%%%%%%%%%
\author[N.S. Papageorgiou]{Nikolaos S. Papageorgiou}
\address[N.S. Papageorgiou]{National Technical University, Department of Mathematics,
				Zografou Campus, Athens 15780, Greece }
\email{\tt npapg@math.ntua.gr}

\author[V.D. R\u{a}dulescu]{Vicen\c{t}iu D. R\u{a}dulescu}
\address[V.D. R\u{a}dulescu]{Department of Mathematics, Faculty of Sciences,
King Abdulaziz University, P.O. Box 80203, Jeddah 21589, Saudi Arabia \& Department of Mathematics, University of Craiova, Street A.I. Cuza 13,
          200585 Craiova, Romania}
\email{\tt vicentiu.radulescu@imar.ro}

\author[D.D. Repov\v{s}]{Du\v{s}an D. Repov\v{s}}
\address[D.D. Repov\v{s}]{Faculty of Education and Faculty of Mathematics and Physics, University of Ljubljana, SI-1000 Ljubljana, Slovenia}
\email{\tt dusan.repovs@guest.ames.si}


\begin{abstract}
We study a semilinear parametric elliptic equation with superdiffusive reaction and mixed boundary conditions. Using variational methods, together with suitable truncation techniques, we prove a bifurcation-type theorem describing the nonexistence, existence and multiplicity of positive solutions.
\end{abstract}

\maketitle

\section{Introduction}

Let $\Omega\subseteq\RR^N$ be a bounded domain with a $C^2$-boundary $\partial\Omega$ and let $\Sigma_1,\Sigma_2\subseteq\partial\Omega$ be two $(N-1)$-dimensional $C^2$-submanifolds   of $\partial\Omega$ such that
$\partial\Omega=\Sigma_1\cup\Sigma_2$, $\Sigma_1\cap\Sigma_2=\emptyset$, $|\Sigma_1|_{N-1}\in(0,|\partial\Omega|_{N-1})$, and $\overline{\Sigma_1}\cap\overline{\Sigma}_2=\Gamma.$
Here, $|\cdot|_{N-1}$ denotes the $(N-1)$-dimensional Hausdorff (surface) measure and $\Gamma\subset\partial\Omega$ is a $(N-2)$-dimensional $C^2$-submanifold of $\partial\Omega$.
%By $n(\cdot)$ we denote the outward unit normal on $\partial\Omega$.

In this paper, we study the following logistic-type elliptic problem:
\begin{equation}
	\left\{\begin{array}{ll}
		-\Delta u(z)=\lambda u(z)^{q-1}-f(z,u(z))&\mbox{in}\ \Omega,\\
	\di	u|_{\Sigma_1}=0,\ \left.\frac{\partial u}{\partial n}\right|_{\Sigma_2}=0,\ u>0,\ \lambda>0.&
	\end{array}\right\}\tag{$P_{\lambda}$}\label{eqp}
\end{equation}

When $f(z,x)=x^{r-1}$ with $r\in(2,2^*)$, we get the classical logistic equation, which is important in biological models (see Gurtin \& Mac Camy \cite{14}). Depending on the value of $q>1$, we distinguish three cases: (i)
 $1<q<2$ (subdiffusive logistic equation); (ii)
	 $2=q<r$ (equidiffusive logistic equation);
	(iii)  $2<q<r$ (superdiffusive logistic equation).
In this paper, we deal with the third situation (superdiffusive case), which exhibits bifurcation-type phenomena for large values of the parameter $\lambda>0$.

Let
$E_{\Sigma_1}=\{u\in H^1(\Omega):u|_{\Sigma_1}=0\}$.
This space is defined as the closure of $C^1_c(\Omega\cup\Sigma_1)$ with respect to the $H^1(\Omega)$-norm. Since  $|\Sigma_1|_{N-1}>0$, we know that for the space $E_{\Sigma_1}$, the Poincar\'e inequality holds (see Gasinski \& Papageorgiou \cite[Problem 1.139, p. 58]{13}). So, $E_{\Sigma_1}$ is a Hilbert space equipped with the norm
$||u||=||Du||_2.$
Let $\mathcal{A}\in\mathcal{L}(E_{\Sigma_1},E_{\Sigma_1}^*)$ be defined by
$\left\langle A(u),h\right\rangle=\int_{\Omega}(Du,Dh)_{\RR^N}dz$ for all $u,h\in E_{\Sigma_1}.$
We denote by $N_f$ the Nemitsky map associated with $f$, that is,
$N_f(u)(\cdot)=f(\cdot,u(\cdot))$ for all $u\in E_{\Sigma_1}.$

The hypotheses on the perturbation term $f(z,x)$ are the following:

\smallskip
$H(f):$ $f:\Omega\times\RR\rightarrow\RR$ is a Carath\'eodory function such that for almost all $z\in\Omega$, $ f(z,0)=0$, $f(z,x)\geq 0$ for all $x>0$, and
\begin{itemize}
	\item[(i)] $f(z,x)\leq a(z)(1+x^{r-1})$ for almost all $z\in\Omega$ and all $x\geq 0$, with $a\in L^{\infty}(\Omega)$, $2<q<r<2^*;$
	\item[(ii)] $\lim\limits_{x\rightarrow+\infty}\frac{f(z,x)}{x^{q-1}}=+\infty$ uniformly for almost all $z\in\Omega$, and the mapping $x\mapsto\frac{f(z,x)}{x}$ is nondecreasing on $(0,+\infty)$ for almost all $z\in\Omega$;
	\item[(iii)] $0\leq\liminf\limits_{x\rightarrow 0^+}\frac{f(z,x)}{x}\leq\limsup\limits_{x\rightarrow 0^+}\frac{f(z,x)}{x}\leq\hat{\eta}$ uniformly for almost all $z\in\Omega$;
	\item[(iv)] for every $\rho>0$, there exists $\hat{\xi}_{\rho}>0$ such that for almost all $z\in\Omega$ the function $x\mapsto\hat{\xi}_{\rho}x-f(z,x)$ is nondecreasing on $[0,\rho]$.
\end{itemize}

	The following functions satisfy hypotheses $H(f)$:
	(i) $f(x)=x^{r-1}$ for all $x\geq 0$ with $2<q<r<2^*$;
	(ii) $f(x)=x^{q-1}\left[\ln(1+x)+\frac{1}{q}\ \frac{x}{1+x}\right]$ for all $x\geq 0$, with $2<q<2^*$.

Let
$\mathcal{L}=\{\lambda>0:\mbox{problem}\ (\refeq{eqp})\ \mbox{has a positive solution}\}$ and let
$S(\lambda)$ denote the set of positive solutions of problem (\refeq{eqp}). Let $\lambda_*=\inf\mathcal{L}$ (if $\mathcal{L}=\emptyset$, then $\inf\emptyset=+\infty$).

By a solution of problem \eqref{eqp}, we understand a function $u\in E_{\Sigma_1}$ such that
$u\geq 0$, $u\neq 0$ and $\left\langle A(u),h\right\rangle=\int_{\Omega}[\lambda u^{q-1}-f(z,u)]hdz$ for all $h\in E_{\Sigma_1}.$

We refer to Bonanno, D'Agui \& Papageorgiou \cite{bona}, Filippucci, Pucci \& R\u adulescu \cite{fpr}, and Li, Ruf, Guo \& Niu \cite{lipaper} for related results. We also refer to the monograph by Pucci \& Serrin \cite{18} for more results concerning the abstract setting of this paper.

\section{A Bifurcation-Type Theorem}

%We start with some auxiliary properties.
\begin{prop}\label{prop2}
	If hypotheses $H(f)$ hold, then $S(\lambda)\subseteq C^{1,\alpha}(\Omega)\cap C^{0,\alpha}(\overline{\Omega})$ with $\alpha\in(0,1/2)$. For all $u\in S(\lambda)$ we have
	$u(z)>0$ for all $z\in\Omega$ and $\lambda_*>0.$
\end{prop}
\begin{proof}
	From DiBenedetto \cite{7} and Colorado \& Peral \cite{5}, we know that if
	$u\in S(\lambda)$ then $u\in C^{1,\alpha}(\Omega)\cap C^{0,\alpha}(\overline{\Omega})$ with $\alpha\in(0,1/2).$
	Moreover, using Harnack's inequality, we deduce that if
	$u\in S(\lambda)$ then $u(z)>0$ for all $z\in\Omega.$
	Let $\hat{\lambda}_1$ be the smallest eigenvalue of $-\Delta$ with mixed boundary conditions. From Colorado \& Peral \cite[p. 482]{5}, we know that
	$\hat{\lambda}_1=\inf\left\{\frac{||Du||^2_2}{||u||^2_2}:u\in E_{\Sigma_1}\setminus \{0\}\right\}>0$.
	By $H(f)(i),(iii)$, there exists $\lambda_0>0$ such that
	\begin{equation}\label{eq2}
		\lambda_0x^{q-1}-f(z,x)\leq\hat{\lambda}_1x\ \mbox{for almost all}\ z\in\Omega,\ \mbox{and all}\ x\geq 0
	\end{equation}
	(recall that $2<q<r$).
Let $\lambda\in(0,\lambda_0)$ and suppose that $\lambda\in\mathcal{L}$. Then there exists $u_{\lambda}\in S(\lambda)$ and by using Green's identity, we get
	\begin{equation}\label{eq3}
		A(u_{\lambda})=\lambda u_{\lambda}^{q-1}-N_f(u_{\lambda})\ \mbox{in}\ E^*_{\Sigma_1}.
	\end{equation}	
We act	on (\ref{eq3}) with $u_{\lambda}\in E_{\Sigma_1}$ and obtain
	$	||Du_{\lambda}||^2_2=\lambda||u_{\lambda}||^q_q-\int_{\Omega}f(z,u_{\lambda})u_{\lambda}dz
		<\hat{\lambda}_1||u_{\lambda}||^2_2$ (see (\ref{eq2}) and recall that $\lambda<\lambda_0,u_{\lambda}(z)>0$ for all $z\in\Omega$),
		which contradicts the definition of $\hat{\lambda}_1$. Therefore $\lambda\notin\mathcal{L}$ and we have
		$0<\lambda_0\leq\lambda_*=\inf\mathcal{L}.$
\end{proof}

\begin{prop}\label{prop3}
	If hypotheses $H(f)$ hold, then $\mathcal{L}\neq\emptyset$ and ``$\lambda\in\mathcal{L},\eta>\lambda\Rightarrow \eta\in\mathcal{L}$".
\end{prop}
\begin{proof}
	Fix $\lambda>0$ and let $\varphi_{\lambda}:E_{\Sigma_1}\rightarrow\RR$,
	$\varphi_{\lambda}(u)=\frac{1}{2}||Du|^2_2-\frac{\lambda}{q}||u^+||^q_q+\int_{\Omega}F(z,u)dz$,
	where $F(z,x)=\int^x_0f(z,s)ds$. Then $\varphi_{\lambda}\in C^1(E_{\Sigma_1})$ and $\varphi_{\lambda}$ is sequentially weakly lower semicontinuous.	
	Hypotheses $H(f)(i),(ii)$ imply that given $\xi>0$, we can find $c_1=c_1(\xi)>0$ such that
	$F(z,x)\geq\frac{\xi}{q}x^q-c_1$ for almost all $z\in\Omega$ and for all $x\geq 0$.
	Thus, for all $u\in E_{\Sigma_1}$  we have
	$\varphi_{\lambda}(u)\geq\frac{1}{2}||Du||^2_2+\frac{\xi-\lambda}{2}||u^+||^q_q-c_1|\Omega|_N$.
	Choosing $\xi>\lambda$, we deduce  that $\varphi_{\lambda}$ is coercive. So, by the Weierstrass-Tonelli theorem, there exists $u_{\lambda}\in E_{\Sigma_1}$ such that
	\begin{equation}\label{eq6}
		\varphi_{\lambda}(u_{\lambda})=\inf\{\varphi_{\lambda}(u):u\in E_{\Sigma_1}\}=m_{\lambda}.
	\end{equation}
		Fix $\bar{u}\in E_{\Sigma_1}\cap C(\overline{\Omega})$ with $u(z)>0$ for all $z\in\Omega$. For large enough $\lambda>0$ we have
		$\varphi_{\lambda}(\bar{u})<0$, hence
			$\varphi_{\lambda}(u_{\lambda})=m_{\lambda}<0=\varphi_{\lambda}(0)$ (see (\ref{eq6})). Thus,
			$ u_{\lambda}\neq 0$.
				By (\ref{eq6}), $\varphi_{\lambda}'(u_{\lambda})=0$, hence
		\begin{equation}\label{eq7}
			A(u_{\lambda})=\lambda(u^+_{\lambda})^{q-1}-N_f(u_{\lambda})\ \mbox{in}\ E^*_{\Sigma_1}.
		\end{equation}
		We act on (\ref{eq7}) with $-u^-_{\lambda}\in E_{\Sigma_1}$ and obtain
		$||Du^-_{\lambda}||^2_2=0$, hence
			$u_{\lambda}\geq 0$.
		So, relation (\ref{eq7}) becomes
		$A(u_{\lambda})=\lambda u^{q-1}_{\lambda}-N_f(u_{\lambda})$. By Green's identity,
			$u_{\lambda}\in S(\lambda)$, hence
		 $\lambda\in\mathcal{L}\neq\emptyset$.
		
		Next, let $\lambda\in\mathcal{L}$ and $\eta>\lambda$. Choose $\vartheta\in(0,1)$ such that $\lambda=\vartheta^{q-2}\eta$ (recall that $2<q$). Also, let $u_{\lambda}\in S(\lambda)\subseteq C^{1,\alpha}(\Omega)\cap C^{0,\alpha}(\overline{\Omega})$ with $\alpha\in(0,1/2)$. Let $\underline{u}=\vartheta u_{\lambda}$. Then
		\begin{equation}\label{eq8}
			A(\underline{u})=\vartheta A(u_{\lambda})=\vartheta\left[\lambda u^{q-1}_{\lambda}-N_f(u_{\lambda})\right]\ \mbox{in}\ E^*_{\Sigma_1}.
		\end{equation}
		
		From hypothesis $H(f)(ii)$ and since $u_{\lambda}(z),\underline{u}(z)>0$ for all $z\in\Omega$, we have for a.a. $z\in\Omega$
		\begin{equation}\label{eq9}
			\frac{f(z,\underline{u}(z))}{\underline{u}(z)}\leq\frac{f(z,u_{\lambda}(z))}{u_{\lambda}(z)}
			\Rightarrow f(z,\underline{u}(z))\leq\vartheta f(z,u_{\lambda}(z))\ (\mbox{recall that}\ \underline{u}=\vartheta u_{\lambda}).
		\end{equation}
		
		Using (\ref{eq8}) in (\ref{eq9}) and since $\vartheta\in(0,1)$, we obtain
		\begin{equation}\label{eq10}
			A(\underline{u})\leq\vartheta^{q-1}\eta u_{\lambda}^{q-1}-N_f(\underline{u})\leq \eta \underline{u}^{q-1}-N_f(\underline{u})\ \mbox{in}\ E^*_{\Sigma_1}.
		\end{equation}
		
		We introduce the following Carath\'eodory truncation of the reaction term in problem ($P_{\eta}$)
		\begin{equation}\label{eq11}
			g_{\eta}(z,x)=\left\{\begin{array}{ll}
				\eta\underline{u}(z)^{q-1}-f(z,\underline{u}(z))&\mbox{if}\ x\leq\underline{u}(z)\\
				\eta x^{q-1}-f(z,x)&\mbox{if}\ \underline{u}(z)<x.
			\end{array}\right.
		\end{equation}	
		 Let $G_{\eta}(z,x)=\int^x_0g_{\eta}(z,s)ds$ and define $\hat{\varphi}_{\eta}:E_{\Sigma_1}\rightarrow\RR$  by
		$\hat{\varphi}_{\eta}(u)=\frac{1}{2}||Du||^2_2-\int_{\Omega}G_{\eta}(z,u)dz$.
		
		Hypotheses $H(f)(i),(ii)$ imply that given $\xi>0$, we can find $c_2=c_2(\xi)>0$ such that
		\begin{equation}\label{eq12}
			\eta x^{q-1}-f(z,x)\leq (\eta-\xi)x^{q-1}+c_2\ \mbox{for almost all}\ z\in\Omega\ \mbox{and all}\ x\geq 0.
		\end{equation}
		
		Then for all $u\in E_{\Sigma_1}$, we have
		\begin{eqnarray}\label{eq13}
			&&\hat{\varphi}_{\eta}(u)\geq \frac{1}{2}||Du||^2_2+\frac{\xi-\eta}{q}||u^+||^q_q-c_3\ \mbox{for some}\ c_3>0\ (\mbox{see (\ref{eq11}), (\ref{eq12})}).
		\end{eqnarray}	
		Choosing $\xi>\eta$, we see  from (\ref{eq13}) that
		$\hat{\varphi}_{\eta}$ is coercive.
		This function is also sequentially weakly lower semicontinuous. So, by the Weierstrass-Tonelli theorem, there exists $u_{\eta}\in E_{\Sigma_1}$ such that $\hat{\varphi}_{\eta}(u_{\eta})=\inf[\hat{\varphi}_{\eta}(u):u\in E_{\Sigma_1}]$, hence $\hat{\varphi}'_{\eta}(u_{\eta})=0$. We deduce that
		\begin{equation}\label{eq14}
			A(u_{\eta})=N_{g_{\eta}}(u_{\eta})\ \mbox{in}\ E_{\Sigma_1}^*.
		\end{equation}
		
		We act on (\ref{eq14}) with $(\underline{u}-u_{\eta})^+\in E_{\Sigma_1}$. By (\ref{eq11}) and (\ref{eq10}) we have
		\begin{eqnarray*}
			\left\langle A(u_{\eta}),(\underline{u}-u_{\eta})^+\right\rangle&=&\int_{\Omega}[\eta\underline{u}^{q-1}-
f(z,\underline{u})](\underline{u}-u_{\eta})^+dz\ \geq\left\langle A(\underline{u}), (\underline{u}-u_{\eta})^+\right\rangle
		\end{eqnarray*}
		\begin{equation}\label{eq15}
			\Rightarrow\left\langle A(\underline{u}-u_{\eta}),(\underline{u}-u_{\eta})^+\right\rangle\leq 0\Rightarrow||D(\underline{u}-u_{\eta})^+||^2_2\leq 0\Rightarrow\underline{u}\leq u_{\eta}.
		\end{equation}
		
		Using (\ref{eq11}) and (\ref{eq15}) we see that relation (\ref{eq14}) becomes
		$A(u_{\lambda})=\eta u^{q-1}_{\eta}-N_f(u_{\eta})$ in $E^*_{\Sigma_1}$. Thus, by Proposition \ref{prop2}, we have
			$u_{\eta}\in S(\eta)\subseteq C^{1,\alpha}(\Omega)\cap C^{0,\alpha}(\overline{\Omega})$.
		Therefore $\eta\in\mathcal{L}$. We also observe that Proposition \ref{prop3} implies $(\lambda_*,+\infty)\subseteq\mathcal{L}.$
\end{proof}

\begin{prop}\label{prop4}
	If hypotheses $H(f)$ hold and $\lambda>\lambda_*$, then problem \eqref{eqp} has at least two positive solutions $u_0,\hat{u}\in E_{\Sigma_1}\cap C^{0,\alpha}(\overline{\Omega})$ for  $\alpha\in(0,1/2)$ with $0<u_0(z),\hat{u}(z)$ for all $z\in\Omega$.
\end{prop}
\begin{proof}
	Let $\mu\in(\lambda_*,\lambda)$. By Proposition \ref{prop3} we know that $\mu\in\mathcal{L}$. Hence we can find $u_{\mu}\in S(\mu)\subseteq E_{\Sigma_1}\cap C^{0,\alpha}(\overline{\Omega})$ with $\alpha\in(0,1/2),u_{\mu}(z)>0$ for all $z\in\Omega$. We have
	$A(u_{\mu})=\mu u^{q-1}_{\mu}-N_f(u_{\mu})$ in $E^*_{\Sigma_1}$.
	Next, we define the following Carath\'eodory function
	\begin{equation}\label{eq17}
		\hat{h}_{\lambda}(z,x)=\left\{\begin{array}{ll}
			\lambda u_{\mu}(z)^{q-1}-f(z,u_{\mu}(z))&\mbox{if}\ x\leq u_{\mu}(z)\\
			\lambda x^{q-1}-f(z,x)&\mbox{if}\ u_{\mu}(z)<x.
		\end{array}\right.
	\end{equation}
	Let $\hat{H}_{\lambda}(z,x)=\int^x_0\hat{h}_{\lambda}(z,s)ds$ and let $\hat{\psi}_{\lambda}:E_{\Sigma_1}\rightarrow\RR$,
		$\hat{\psi}_{\lambda}(u)=\frac{1}{2}||Du||^2_2-\int_{\Omega}\hat{H}_{\lambda}(z,u)dz$.
		Then $\hat{\psi}_{\lambda}$ is coercive and sequentially weakly lower semicontinuous. Thus, we can find $u_0\in E_{\Sigma_1}$ such that $\hat{\psi}_{\lambda}(u_0)=\inf\{\hat{\psi}_{\lambda}(u):u\in E_{\Sigma_1}\}$, hence
		$\hat{\psi}'_{\lambda}(u_0)=0$. Thus,
			$A(u_0)=N_{\hat{h}_{\lambda}}(u_0)$. Using (\ref{eq17}) and reasoning as in the proof of Proposition \ref{prop3}
			we deduce that $u_{\mu}\leq u_0$.
		By Colorado \& Peral \cite[Theorem 6.6]{5}, we have
		$u_0\in E_{\Sigma_1}\cap C^{0,\alpha}(\overline{\Omega})$ with $\alpha\in(0,1/2)$
		and $u_0>0$ in $z\Omega$ (by Harnack's inequality).		
		
		Let $\rho_0=||u_0||_{\infty}$ and let $\hat{\xi}_{\rho_0}>0$ be as postulated in hypothesis $H(f)(iv)$. We have
		\begin{equation}\label{eq18}
			\left\{\begin{array}{l}-\Delta u_0(z)+\hat{\xi}_{\rho_0}u_0(z)=\lambda u_0(z)^{q-1}-f(z,u_0(z))+\hat{\xi}_{\rho_0}u_0(z)\ \mbox{in}\ \Omega,\\
		\di	u_0|_{\Sigma_1}=0,\left.\frac{\partial u_{0}}{\partial n}\right|_{\Sigma_2}=0
			\end{array}\right\}
		\end{equation}
and
		\begin{equation}\label{eq19}
			\left\{\begin{array}{l}
				-\Delta u_{\mu}(z)+\hat{\xi}_{\rho_0}u_{\mu}(z)=\mu u_{\mu}(z)^{q-1}-f(z,u_{\mu}(z))+\hat{\xi}_{\rho_0}u_{\mu}(z)\ \mbox{in}\ \Omega,\\
	\di			\hat{u}_{\mu}|_{\Sigma_1}=0,\left.\frac{\partial u_{\mu}}{\partial n}\right|_{\Sigma_2}=0.
			\end{array}\right\}
		\end{equation}
		Let $\hat{y}=u_0-u_{\mu}\geq 0$. Since $\lambda>\mu,u_0\geq u_{\eta}$, from (\ref{eq18}), (\ref{eq19}), and $H(f)(iv)$ we have
		\begin{eqnarray*}
			&&-\Delta\hat{y}(z)+\hat{\xi}_{\rho_0}\hat{y}(z)=\lambda u_0(z)^{q-1}-\mu u_{\mu}(z)^{q-1}+[\hat{\xi}_{\rho_0}u_0(z)-f(z,u_0(z))]-\\
			&&-[\hat{\xi}_{\rho_0}u_{\mu}(z)-f(z,u_{\mu}(z))]\geq 0\ \mbox{in}\ \Omega.
		\end{eqnarray*}	
		
		Let $v_1\in E_{\Sigma_1}$ be the unique function satisfying
		$-\Delta v(z)+\hat{\xi}_{\rho_0}v(z)=1$ $\Omega$,
			$v|_{\Sigma_1}=0$, and $\left.\frac{\partial v}{\partial n}\right|_{\Sigma_2}=0$.
		Then $v_1\in C^{1,\alpha}(\Omega)\cap C^{0,\alpha}(\overline{\Omega})$ with $\alpha\in(0,1/2)$ (see \cite{5,7}) and $v_1>0$ in $\Omega$. By Lemma 2.1 of Barletta, Livrea \& Papageorgiou \cite{3} (see also Lemma 5.3 of Colorado \& Peral \cite{5}), we can find $\vartheta>0$ such that
		\begin{equation}\label{eq20}
			\vartheta v_1(z)\leq u_{\mu}(z)\ \mbox{and}\ \vartheta v_1(z)\leq \hat{y}(z)\Rightarrow\vartheta v_1(z)\leq u_{\mu}(z)\leq u_0(z)-\vartheta v_1(z)\ \mbox{for all}\ z\in\overline{\Omega}.
		\end{equation}
		
		Let
		$\hat{C}_1=\left\{y\in E_{\Sigma_1}\cap C(\overline{\Omega}):\left\|\frac{y}{v_1}\right\|_{\infty}<\infty\right\}$ and
		 $\left[u_{\mu}\right)=\{u\in E_{\Sigma_1}:u_{\mu}\leq u(z),\ \mbox{a.a.}\ z\in\Omega\}$. We claim that
		if $\bar{B}_1(0):=\{y\in\hat{C}_1:\left\|\frac{y}{v_1}\right\|_{\infty}\leq 1\}$,
			then $u_0-\vartheta\bar{B}_1(0)\subseteq\left[u_{\mu}\right)\cap \hat{C}_1$.
		To see this, let $y\in\bar{B}_1(0)$. Then
		\begin{equation}\label{eq22}
		-v_1(z)\leq y(z)\leq v_1(z)\ \mbox{for all}\ z\in\overline{\Omega}.
		\end{equation}
		
		Fix $z\in\overline{\Omega}$. If $y(z)>0$, then
		$0\leq u_{\mu}(z)\leq u_{\mu}(z)+\vartheta y(z)\leq u_{\mu}(z)+\vartheta v_1(z)\leq u_0(z)$ (see (\ref{eq20}), (\ref{eq22})), hence
			$u_{\mu}(z)\leq u_0(z)-\vartheta y(z)$.	
		If $y(z)<0$, then
		$0\leq u_{\mu}(z)-\vartheta v_1(z)\leq u_{\mu}(z)+\vartheta y(z)\leq u_{\mu}(z)+\vartheta v_1(z)\leq u_0(z)$ (see (\ref{eq20}), (\ref{eq22})), hence
			$u_{\mu}(z)\leq u_0(z)-\vartheta y(z)$.
		We conclude that
		$u_{\mu}\in u_0-\vartheta\bar{B}_1(0),$
		which argues our above claim. It follows   that
		\begin{equation}\label{eq23}
			u_0\in {\rm int}_{\hat{C}_1}\left[u_{\mu}\right)\cap C(\overline{\Omega}).
		\end{equation}
		
		By (\ref{eq17}) it is clear that
		\begin{equation}\label{eq24}
			\hat{\psi}_{\lambda}(u)=\varphi_{\lambda}(u)+c_4\ \mbox{for some}\ c_4\in\RR\ \mbox{and for all}\ u\in\left[u_{\mu}\right).
		\end{equation}
		
		It follows from (\ref{eq23}) and (\ref{eq24}) that $u_0$ is a local $\hat{C}_1$-minimizer of $\varphi_{\lambda}$.
		
{\bf Claim.}
		{\it	$u_0$ is a local $E_{\Sigma_1}$-minimizer of $\varphi_{\lambda}$.}
		Suppose that this assertion is not true. Then for every $\rho>0$, we have
		$\inf\{\varphi_{\lambda}(u_0+y):y\in E_{\Sigma_1},\ ||y||\leq\rho\}<\varphi_{\lambda}(u_0).$
		By the Weierstrass-Tonelli theorem, there exists $y_{\rho}\in E_{\Sigma_1}\backslash\{0\},\ ||y_{\rho}||\leq\rho$ such that
		$\varphi_{\lambda}(u_0+y_{\rho})=\inf\{\varphi_{\lambda}(u_0+y):y\in E_{\Sigma_1},||y||\leq\rho\}<\varphi_{\lambda}(y_0).$
		By the Lagrange multiplier rule, there exists $\vartheta\leq 0$ such that
		$(1-\vartheta)\left\langle A(u_{\rho}),h\right\rangle=\lambda\int_{\Omega}(u^+_{\rho})^{q-1}hdz-\int_{\Omega}f(z,u_{\rho})hdz$ for all $h\in E_{\Sigma_1}$,
		with $u_{\rho}=u_0+y_{\rho}\in E_{\Sigma_1}$. It follows that $\Delta u_{\rho}(z)=\frac{1}{1-\vartheta}[\lambda u^+_{\rho}(z)^{q-1}-f(z,u_{\rho}(z))]$ in $\Omega$, hence
		\begin{equation}\label{eq25}
			-\Delta u_{\rho}(z)+\hat{\xi}_{\rho_0}u_{\rho}(z)=\frac{1}{1-\vartheta}[\lambda u^+_{\rho}(z)^{q-1}+f(z,u_{\rho}(z))]+\hat{\xi}_{\rho_0}u_{\rho}(z)\ \mbox{in}\ \Omega,	
		\end{equation}
		with $\hat{\xi}_{\rho_0}>0$ as before resulting from hypothesis $H(f)(iv)$ (recall that $\rho_0=||u_0||_{\infty}$).	
		Also,
		\begin{equation}\label{eq26}
			-\Delta u_0(z)+\hat{\xi}_{\rho_0}u_0(z)=\lambda u_0(z)^{q-1}-f(z,u_0(z))+\hat{\xi}_{\rho_0}u_0(z)\ \mbox{in}\ \Omega.
		\end{equation}
		
		From (\ref{eq25}) and (\ref{eq26}) we obtain
		\begin{equation}\label{eq27}
			-\Delta y_{\rho}(z)+\hat{\xi}_{\rho_0}y_{\rho}(z)=g^{\rho}_{\lambda}(z)\ \mbox{in}\ \Omega
		\end{equation}
		with $g^{\rho}_{\lambda}(z)=\frac{1}{1-\vartheta}[\lambda u^+_{\rho}(z)^{q-1}-f(z,u_{\rho}(z))]-\lambda u_0(z)^{q-1}+f(z,u_0(z))+\hat{\xi}_{\rho_0}y_{\rho}(z).$
By (\ref{eq27}) and Colorado \& Peral \cite{5}, there exist $c_5>0$ and $\alpha\in(0,1/2)$ such that
		\begin{equation}\label{eq28}
			y_{\rho}\in C^{0,\alpha}(\overline{\Omega})\ \mbox{and}\ ||y_{\rho}||_{C^{0,\alpha}(\overline{\Omega})}\leq c_5\ \mbox{for all}\ \rho\in\left(0,1\right].
 		\end{equation}
		
		Exploiting the compact embedding of $C^{0,\alpha}(\overline{\Omega})$ into $C(\overline{\Omega})$, we have
		$y_{\rho}\rightarrow 0\ \mbox{in}\ C(\overline{\Omega})$ as $\rho\rightarrow 0^+$. Thus,
		by the definition of $g^{\rho}_{\lambda}$, there exists $\tau^*_{\rho}>0$ such that
		\begin{equation}\label{eq30}
			||g^{\rho}_{\lambda}||_{\infty}\leq\tau^*_{\rho}\ \mbox{for all}\ \rho\in\left(0,1\right]\ \mbox{and}\ \tau^*_{\rho}\rightarrow 0^+\ \mbox{as}\ \rho\rightarrow 0^+.
		\end{equation}
		
		Let $\hat{y}_{\rho}=\frac{1}{\tau^*_{\rho}}y_{\rho}$. Then by (\ref{eq30})
		$-\Delta (\hat{y}_{\rho}-v_1)(z)+\hat{\xi}_{\rho_0}(\hat{y}_{\rho}-v_1)(z)=\frac{1}{\tau^*_{\rho}}g^{\rho}_{\lambda}(z)-1\leq 0$. We deduce that
			$||D(\hat{y}_{\rho}-v_1)^+||^2_2+\hat{\xi}_{\rho_0}||(\hat{y}_{\rho}-v_1)^+||^2_2\leq 0$, hence
			$ y_{\rho}\leq \tau^*_{\rho}v_1$.
		
		Also, we have
		$-\Delta (-\hat{y}_{\rho}-v_1)(z)+\hat{\xi}_{\rho_0}(-\hat{y}_{\rho}-v_1)(z)=-\frac{1}{\tau^*_{\rho}}g^{\rho}_{\lambda}(z)-1\leq 0$ in $\Omega$
		and so as above we obtain that
		$-\tau^*_{\rho}v_1\leq y_{\rho}.$
		Therefore we have proved that
		$-\tau^*_{\rho}v_1\leq y_{\rho}\leq \tau^*_{\rho}v_1$.
		These relations show that
		$y_{\rho}\in \hat{C}_1$ and $\left\|\frac{y_{\rho}}{v_1}\right\|_{\infty}\leq\tau^*_{\rho}$ for all $\rho\in\left(0,1\right]$, hence
			$y_{\rho}\rightarrow 0$ in $\hat{C}_1$ as $\rho\rightarrow 0^+$.
		Therefore for small $\rho\in\left(0,1\right]$ we have
		$\varphi_{\lambda}(u_0+y_{\rho})<\varphi_{\lambda}(u_0),$
		which contradicts the fact that $u_0$ is a local $\hat{C}_1$-minimizer of $\varphi_{\lambda}$. This proves the claim.
		
		Since $f\geq 0$, for all $u\in E_{\Sigma_1}$ we have
$\varphi_{\lambda}(u)\geq\frac{1}{2}||Du||^2_2-\frac{\lambda}{q}||u^+||^q_q
			\geq\frac{1}{2}||Du||^2_2-c_6||Du||^q_2$ for some $c_6>0$.
	%	\begin{equation}\label{eq32}
	%		\varphi_{\lambda}(u)\geq\frac{1}{2}||Du||^2_2-\frac{\lambda}{q}||u^+||^q_q
	%		\geq\frac{1}{2}||Du||^2_2-c_6||Du||^q_2\ \mbox{for some}\ c_6>0
	%	\end{equation}
	%	(recall that $E_{\Sigma_1}$ is compactly embedded into $L^q(\Omega)$).	
		Since $q>2$, we deduce that $u=0$ is a local minimizer of $\varphi_{\lambda}$. We assume that the set of critical points of $\varphi_{\lambda}$ is finite (otherwise we already have an infinity of positive solutions for \eqref{eqp} for $\lambda>\lambda_*$ and so we are done) and that $\varphi_{\lambda}(0)\leq\varphi_{\lambda}(u_0)$ (the reasoning is similar if the opposite inequality holds). The claim implies that we can find small enough $\rho\in(0,||u_0||)$ such that
$0=\varphi_{\lambda}(0)\leq\varphi_{\lambda}(u)<\inf\{\varphi_{\lambda}(u):||u-u_0||=\rho\}=m^{\rho}_{\lambda}$.
%		\begin{equation}\label{eq33} %0=\varphi_{\lambda}(0)\leq\varphi_{\lambda}(u)<\inf\{\varphi_{\lambda}(u):||u-u_0||=\rho\}=m^{\rho}_{\lambda}.
%		\end{equation}
	%	(see Aizicovici, Papageorgiou and Staicu \cite{1}, proof of Proposition 29).	
	%	Recall that $\varphi_{\lambda}$ is coercive and so
	%	$\varphi_{\lambda}$ satisfies the C-condition
	%	(see Papageorgiou and Winkert \cite[Proposition 2.13]{17}).
Thus, we can apply the mountain pass theorem. So, there exists $\hat{u}\in E_{\Sigma_1}$ such that
		$\varphi'_{\lambda}(\hat{u})=0$ and $m^{\rho}_{\lambda}\leq\varphi_{\lambda}(\hat{u})$, hence
			$\hat{u}\notin\{0,u_0\}$, $\hat{u}\in S_{\lambda}\subseteq E_{\Sigma_1}\cap C^{0,\alpha}(\overline{\Omega})$, and $\hat{u}>0$ in $\Omega$.
\end{proof}

%Next, we examine what happens in the critical case that corresponds to $\lambda=\lambda_*$.
\begin{prop}\label{prop5}
	If hypotheses $H(f)$ hold, then $\lambda_*\in\mathcal{L}$, that is, $\mathcal{L}=\left[\lambda^*,+\infty\right)$.
\end{prop}
\begin{proof}
	Let $\{\lambda_n\}_{n\geq 1}\subseteq(\lambda_*,+\infty)$ be such that $\lambda_n\downarrow\lambda_*$. We find $u_n\in S(\lambda_n)$ such that
	\begin{equation}\label{eq35}
		A(u_n)=\lambda u^{q-1}_n-N_f(u_n)\ \mbox{in}\ E^*_{\Sigma_1}\ \mbox{for all}\ n\in\NN.
	\end{equation}
	
	Hypotheses $H(f)(i),(ii)$ imply that given any $\xi>0$, we find $c_7=c_7(\xi)>0$ such that
	\begin{equation}\label{eq36}
		f(z,x)\geq\xi x^{q-1}-c_7\ \mbox{for almost all}\ z\in\Omega\ \mbox{and all}\ x\geq 0.
	\end{equation}
	
We act	on (\ref{eq35}) with $u_n\in E_{\Sigma_1}$ and then use (\ref{eq36}). We obtain
	$||Du||^2_2\leq(\lambda_n-\xi)||u_n||^q_q+c_7|\Omega|_N.$
	Choosing $\xi>\lambda_1\geq\lambda_n$ for all $n\in\NN$, we have
	$||Du_n||^2_2\leq c_7|\Omega|_N$ for all $n\in\NN$, hence
		$\{u_n\}_{n\geq 1}\subseteq E_{\Sigma_1}$ is bounded.
	By passing to a subsequence if necessary, we may assume that
	\begin{equation}\label{eq37}
		u_n\stackrel{w}{\rightarrow}u_*\ \mbox{in}\ E_{\Sigma_1}\ \mbox{and}\ u_n\rightarrow u\ \mbox{in}\ L^r(\Omega)\ \mbox{as}\ n\rightarrow\infty.
	\end{equation}
	In (\ref{eq35}) we pass to the limit as $n\rightarrow\infty$ and use (\ref{eq37}). Then $A(u_*)=\lambda_* u_*^{q-1}-N_f(u_*)$. Thus,
	$u_*\in E_{\Sigma_1}$ and  $u_*\geq 0$ is a solution of  ($P_{\lambda_*}$).
	We also notice that $\lim_{n\rightarrow\infty}\left\langle A(u_n),u_n-u_*\right\rangle=0$, hence $||Du_n||_2\rightarrow||Du_*||_2$. Using the Kadec-Klee property we deduce that $u_n\rightarrow u_*$ in $E_{\Sigma_1}$.
%	\begin{equation}\label{eq38}
%		u_n\rightarrow u_*\ \mbox{in}\ E_{\Sigma_1}.
%	\end{equation}
	
	{\bf Claim.} {\it
		$u_*\neq 0.$}
	Arguing by contradiction, suppose that $u_*=0$. Then
		$||u_n||\rightarrow 0$.
	Let $y_n=\frac{u_n}{||u_n||},\ n\in\NN$. Then $||y_n||=1,\ y_n\geq 0$ for all $n\in\NN$. From (\ref{eq35}) we have
	\begin{equation}\label{eq40}
		A(y_n)=\lambda_nu_n^{q-2}y_n-\frac{N_f(u_n)}{||u_n||}\ \mbox{for all}\ n\in\NN.
	\end{equation}
	
	From hypotheses $H(f)(i),(iii)$, we see that we can find $\eta>\hat{\eta}$ and $c_8>0$ such that
	\begin{equation}\label{eq41}
		f(z,x)\leq\eta x+c_8 x^{r-1}\ \mbox{for a.a.}\ z\in\Omega,\ \mbox{all}\ x\geq 0
		\Rightarrow\{N_f(u_n)\}_{n\geq 1}\subseteq L^2(\Omega)\ \mbox{is bounded}.
	\end{equation}
	
	By \cite{5}, there exist $\alpha\in(0,1/2)$ and $c_9>0$ such that
	$u_n\in C^{0,\alpha}(\overline{\Omega})$, $||u_n||_{C^{0,\alpha}(\overline{\Omega})}\leq c_9$ for all $n\in\NN$.
	Since $C^{0,\alpha}(\overline{\Omega})$ is compactly embedded compactly in $C(\overline{\Omega})$, we deduce that
	\begin{equation}\label{eq43}
		u_n\rightarrow 0\ \mbox{in}\ C(\overline{\Omega}).
	\end{equation}
	
	Recall that $||y_n||=1,\ y_n\geq 0$ for all $n\in\NN$. So, we may assume that
	\begin{equation}\label{eq44}
		y_n\stackrel{w}{\rightarrow}y\ \mbox{in}\ E_{\Sigma_1}\ \mbox{and}\ y_n\rightarrow y\ \mbox{in}\ L^2(\Omega), y\geq 0.
	\end{equation}
	
It follows	from (\ref{eq41}), (\ref{eq43}) and (\ref{eq44}) that
	$\left\{\frac{N_f(u_n)}{||u_n||}\right\}_{n\geq 1}\subseteq L^2(\Omega)$ is bounded.
	Thus, by hypothesis $H(f)(iii)$, we have at least for a subsequence,
	\begin{equation}\label{eq45}
		\frac{N_f(u_n)}{||u_n||}\stackrel{w}{\rightarrow}\eta_0y\ \mbox{in}\ L^2(\Omega)\ \mbox{with}\ 0\leq\eta_0(z)\leq\hat{\eta}\ \mbox{for almost all }z\in\Omega.
	\end{equation}
	We act on (\ref{eq40}) with $y_n-y\in E_{\Sigma_1}$ and  pass to the limit as $n\rightarrow\infty$. Using (\ref{eq43}), (\ref{eq44}) and (\ref{eq45}) we obtain $\lim_{n\rightarrow\infty}\left\langle A(y_n),y_n-y\right\rangle=0$. By the Kadec-Klee property we have $y_n\rightarrow y$, hence
	$||y||=1$, $y\geq 0$.
	In (\ref{eq40}) we pass to the limit as $n\rightarrow\infty$ and use (\ref{eq43}), (\ref{eq45}). Then
	$A(y)=-\eta_0y$. Thus, by (\ref{eq45})  we have
		$||Dy||^2_2=-\int_{\Omega}\eta_0y^2dz\leq 0$, hence
		$y=0$, a contradiction.
	This shows that the claim is true.
	Hence $u_*\in S(\lambda_*)\subseteq E_{\Sigma_1}\cap C(\overline{\Omega})$ and so $\lambda_*\in\mathcal{L}$.
\end{proof}

Summarizing, we can state the following bifurcation-type theorem.
\begin{theorem}\label{th6}
	If hypotheses $H(f)$ hold, then there exists $\lambda_*>0$ such that
	\begin{itemize}
		\item[(a)] for all $\lambda>\lambda_*$, problem \eqref{eqp} has at least two positive solutions
		$u_0,\hat{u}\in E_{\Sigma_1}\cap C(\overline{\Omega})$;
		\item[(b)] for $\lambda=\lambda_*$, problem \eqref{eqp} has at least one positive solution
		$u_*\in E_{\Sigma_1}\cap C(\overline{\Omega})$;
		\item[(c)] for $\lambda\in(0,\lambda_*)$, problem \eqref{eqp} has no positive solutions.
	\end{itemize}
\end{theorem}
	
	\medskip
{\bf Acknowledgments.} This research was supported by the Slovenian Research Agency grants P1-0292, J1-8131, and J1-7025. V.D. R\u adulescu acknowledges the support through a grant of the Romanian National Authority for Scientific Research and Innovation, CNCS-UEFISCDI, project number PN-III-P4-ID-PCE-2016-0130.

\end{document}